\documentclass[12pt,a4paper,oneside,reqno]{amsproc}
\usepackage{amsmath} %amscd
\usepackage{amssymb,amsfonts,mathrsfs}
\usepackage[colorlinks=true,linkcolor=blue,citecolor=blue]{hyperref}
\usepackage[driver=pdftex,lmargin=2.5cm,rmargin=2.5cm,heightrounded=true,centering]{geometry}
\usepackage{math50}
\usepackage{local}
\usepackage{epsfig}  % want eps figures?
\usepackage{graphicx}  % want eps figures?
\usepackage{xy}
\xyoption{all}
\begin{document}
\title[Fractional Schrodinger Equations]
{Concentration phenomenon for fractional nonlinear Schr\"{o}dinger equations}
\author{Guoyuan Chen}
\address{School of Mathematics and Statistics, Zhejiang University of Finance \& Economics, Hangzhou 310018, Zhejiang, P. R. China}
\email{gychen@zufe.edu.cn}
\author{Youquan Zheng$^{\dag}$}
\address{School of Science, Tianjin University, Tianjin 300072, P. R. China.}
\email{zhengyq@tju.edu.cn}
\newcommand{\optional}[1]{\relax}
\setcounter{secnumdepth}{3}
\setcounter{section}{0} \setcounter{equation}{0}
\numberwithin{equation}{section}
\newcommand{\MLversion}{1.1}
\thanks{$\dag$ The first author was partially supported by Zhejiang Provincial Natural Science
Foundation of China (LQ13A010003). The second author was partially
supported by NSFC of China (11271200).}
%\subjclass[2000]{Primary 58J37; Secondary 58J40; 58J50; 58J05}
%\date{\today}
\begin{abstract}
We study the concentration phenomenon for solutions of the fractional nonlinear Schr\"{o}dinger equation, which is nonlocal. We mainly use the Lyapunov-Schmidt reduction method. Precisely, consider the nonlinear equation
\begin{equation}\label{e:abstract}
(-\varepsilon^2\Delta)^sv+Vv-|v|^{\alpha}v=0\quad\mbox{in}\quad\mathbf R^n,
\end{equation}
where $n =1, 2, 3$, $\max\{\frac{1}{2}, \frac{n}{4}\}< s < 1$, $1 \leq \alpha < \alpha_*(s,n)$, $V\in C^3_{b}(\mathbf{R}^n)$. Here the exponent $\alpha_*(s,n)=\frac{4s}{n-2s}$ for $0 < s < \frac{n}{2}$ and $\alpha_*(s,n)=\infty$ for $s \geq\frac{n}{2}$. Then for each non-degenerate critical point $z_0$ of $V$, there is a nontrivial solution of equation (\ref{e:abstract}) concentrating to $z_0$ as $\varepsilon\to 0$.
\end{abstract}
\maketitle

%\markboth{Version: \MLversion, \today}{Version: \MLversion, \today}

\section{Introduction}
We mainly consider the fractional Schr\"{o}dinger equation in this paper.

The Schr\"{o}dinger equation, often called the Schr\"{o}dinger wave equation, is a fundamental equation of quantum mechanics which describes how the wave function of a physical system evolves over time. This equation is not derived from a conical set of axioms. For example, Schr\"{o}dinger himself, arrived at this equation by inserting de Broglies relation into a classical wave equation. Another attempt to derive Schr\"{o}dinger equation from classical physics was using Nelson's stochastic theory \cite{Nelson1985}. Hall and Reginatta \cite{HallReginatto2002} showed that the Schrodinger equation can be derived from the exactly uncertainty principle. It is also well known that Feynman and Hibbs used path integrals over Brownian paths to derive
the standard Schr\"{o}dinger equation \cite{FeynmanHibbs2010}.

It is worthwhile to mention that the above treatments leads to the regular kind of Schr\"{o}dinger equations which are solved by the regular calculus (integro and differential equation). Fractional calculus generalized the classical calculus. Fractional spaces and the corresponding nonlocal equations have many important applications in various fields of science and engineering. For example, the thin obstacle problem \cite{MilakisSilvestre08, Silvestre07}, optimization \cite{DuvautLions76}, finance \cite{ContTankov04}, phase transition \cite{AlbertBouchittSeppecher98, CabreMorales2005, SireValdinoci2009, FarinaValdinici2011, GarroniPalatucci2006}, stratified materials \cite{SavinValdinicci2009}, anomalous diffusion \cite{MetzlerKlafter2000}, crystal dislocation \cite{JFToland1997, GonzalezMonneau2012}, soft thin films \cite{Kurzke2006}, semipermeable membranes and flame propagation \cite{CaffarelliMelletSire2012}, conservation laws \cite{BilerKarchWoyczyski2001}, ultra-relativistic limits of quantum mechanics \cite{FeffermanDelallave1986}, quasi-geostrophic flows \cite{MajdaTabak1996, Cordoba1998}, multiple scattering \cite{DuistermaatGuillemin1975, ColtonKress1998, GroteKirsch2004}, minimal surfaces \cite{CaffarelliRoquejoffreSavin, CaffarelliValdinoci2011}, conformal geometry \cite{ChangGonzalez2011}, material science \cite{Bates2006} and water waves \cite{stokerJJ1957, whiteham1974, CraigSulem1992, CraigGroves1994, CraigNicholls2000, CraigWorfolk1995, CraigSchanzSulem1997, DelallavePanayotaros1996, DellallaveValdinoci2009, GachterGrote2003, HuNicholls2005, NichollsTaber2008, NaumkinShishmar1994} and so on.

In quantum physics, the Feynman path integral
approach to quantum mechanics was the first successful attempt applying the fractality concept
that was first introduced by Mandelbrot \cite{Mandelbrot1982}. Recently, Laskin extended the fractality concept and
formulated fractional quantum mechanics as a path integral over the L\'{e}vy flights paths \cite{Laskin2000a, Laskin2000b, Laskin2000c, Laskin2002}. Through
introducing the quantum Riesz fractional derivative, he constructed the space fractional Schr\"{o}dinger equation
\begin{equation*}
i\hbar\frac{\partial\psi}{\partial t}=D_{\alpha}(-\hbar^2\Delta)^{\frac{\alpha}{2}}\psi+V(x, t)\psi.
\end{equation*}
Laskin showed the hermiticity of the fractional Hamilton
operator and established the parity conservation law. Energy spectra of a hydrogen like atom and of a fractional oscillator were also computed.

Mathematically, $(-\Delta)^s$ is defined as
$$
(-\Delta)^s u = C(n, s)\mbox{P. V.} \int_{\mathbf{R}^n}\frac{u(x) - u(y)}{|x - y|^{n + 2s}}dy = C(n, s)\lim_{\varepsilon\to 0^+}\int_{B_{\varepsilon}^c(x)}\frac{u(x) - u(y)}{|x - y|^{n + 2s}}dy.
$$
Here P. V. is a commonly used abbreviation for `in the principal value sense' and $C(n, s) = \pi^{-(2s + n/2)}\frac{\Gamma(n/2 + s)}{\Gamma(-s)}$.
It is well known that $(-\Delta)^s$ on $\mathbf{R}^{n}$ with $s\in (0, 1)$ is a nonlocal operator. In the remarkable work of Caffarelli and Silvestre \cite{Caffarelli&Silvestre07}, the authors express this nonlocal operator as a generalized Dirichlet-Neumann map for a certain elliptic boundary value problem with local differential operators defined on the upper half-space $\mathbf{R}^{n+1}_{+} = \{(x, t): x\in\mathbf{R}^{n}, t > 0\}$. That is, given a solution $u = u(x)$ of $(-\Delta)^s u = f$ in $\mathbf{R}^{n}$, one can equivalently consider the dimensionally extended problem for $u = u(x, t)$, which solves
\begin{equation}\label{e:dirichletneumann}
\begin{cases} \text{div}(t^{1 -2s}\nabla u) = 0 , &  \text{ in }\mathbf{R}^{n+1}_{+},\\
-d_{s}t^{1 -2s}\partial_{t}u|_{t\to 0} = f, & \text{ on $\partial\mathbf{R}^{n+1}_{+}$}.
\end{cases}
\end{equation}
Here the positive constant $d_{s} > 0$ is explicitly given by
$$
d_{s} = 2^{2s -1}\frac{\Gamma(s)}{\Gamma(1 - s)}.
$$
The formulation (\ref{e:dirichletneumann}) in terms of local differential operators plays a central role when deriving bounds on the number of sign changes for eigenfunctions of fractional Schr\"{o}dinger operators $H = (-\Delta)^s + V$. Using this idea, \cite{Frank&LenzmannActaMath} and \cite{FrankLenzmann&Silvestre}
obtain certain sharp oscillation bounds for eigenfunctions of $H$. In these two papers, the authors proved the uniqueness and nondegeneracy ground state solutions $0\leq u_0 = u_0(|x|)\in H^{s}(\mathbf{R}^n)$ for the nonlinear problem
\begin{equation}\label{e:limit}
(-\Delta)^su + u = |u|^{\alpha}u\quad\mbox{in}\quad \mathbf{R}^n,
\end{equation}
which settled a conjecture by \cite{KenigMartelRobbiano2011, WeinsteinCPDE87} and generalized a classical result by Amick and Toland \cite{Amick&Toland91} on the uniqueness of solitary waves for the Benjamin-Ono equation. If $s = 1$, the uniqueness and nondegeneracy for ground states of (\ref{e:limit}) was due to \cite{KwongARMA1989}.

In this paper, we consider fractional Schr\"{o}dinger equation
\begin{equation}\label{e:evolutionEqn}
i\varepsilon \frac{\partial\psi}{\partial t}=(-\varepsilon^2\Delta)^s\psi+V\psi-\gamma|\psi|^{\alpha}\psi \quad\mbox{in}\quad\mathbf R^n,
\end{equation}
where $\varepsilon$ is the sufficiently small positive constant which is corresponding to the Plank constant, $\gamma>0$, $n\ge 1$, $s\in (0,1)$, $\alpha\in(0,\alpha_*(s,n))$. The exponent $\alpha_*(s,n)$ satisfies
\begin{equation}
\alpha_*(s,n)=\begin{cases} \frac{4s}{n-2s} , &  \text{ for }0<s<\frac{n}{2},\\
+\infty, & \text{ for $s\ge \frac{n}{2}$}.
\end{cases}
\end{equation}

We shall seek the so-called solitary waves which is of form
\begin{equation}\label{e:standingwave}
\psi(x,t)=e^{(-i/\varepsilon)Et}v(x),
\end{equation}
where $v$ is real-valued and $E$ is some constant in $\mathbf R$. (\ref{e:standingwave}) solves (\ref{e:evolutionEqn}) provided the standing wave $v(x)$ satisfy the nonlinear eigenvalue equation
\begin{equation}\label{e:stationally}
(-\varepsilon^2\Delta)^sv+Vv-\gamma |v|^{\alpha}v=Ev\quad\mbox{in}\quad\mathbf R^n,
\end{equation}
For simplicity of the notation, we shall assume that $\gamma=1$, so Equation (\ref{e:stationally}) is reduced to
\begin{equation}
(-\varepsilon^2\Delta)^sv+(V-E)v- |v|^{\alpha}v=0\quad\mbox{in}\quad\mathbf R^n,
\end{equation}
Since we assume that $V\in C^3_{b}(\mathbf{R}^n)$, where
\begin{equation*}
C_{b}^3(\mathbf{R}^n) : = \{u\in C^3(\mathbf{R}^n): D^J u \mbox{\,\,is bounded on\,\,} \mathbf{R}^n \mbox{\,\,for\,\,} |J|\leq 3\}
\end{equation*}
with norm
$\|u\|_{C^3_{b}(\mathbf{R}^n)} = \displaystyle\max_{0\leq |J|\leq 3}\sup_{x\in \mathbf{R}^n}|D^J(x)|$, by a suitable choice of $E$ we can assume that $V-E$ is positive. Finally, we have the following equation
\begin{equation}\label{e:main}
(-\varepsilon^2\Delta)^sv+Vv-|v|^{\alpha}v=0\quad\mbox{in}\quad\mathbf R^n,
\end{equation}
where $V$ is a positive function.

The main result of this paper is
\begin{theorem}\label{t:main}
Assume that $n =1, 2, 3$, $\max\{\frac{1}{2}, \frac{n}{4}\}< s < 1$, $1 \leq \alpha < \alpha_{*}(s, n)$, $V\in C^3_{b}(\mathbf{R}^n)$. Then for each non-degenerate critical point $z_0$ of $V$, there is an $\varepsilon_0>0$ such that, for all $\varepsilon\in (0,\varepsilon_0)$, equation (\ref{e:main}) has a nontrivial solution $v_\varepsilon$ concentrating to $z_0$ as $\varepsilon\to 0$. More precisely,
$v_\varepsilon$ takes the form
\begin{equation}
v_\varepsilon(x) = u_{0}\left(\frac{x - z_\varepsilon}{\varepsilon}\right) + \phi_{z_\varepsilon,\varepsilon}
\end{equation}
with $z_\varepsilon \to z_0$, $\phi_{z_\varepsilon,\varepsilon}\to 0$ in $H^{2s}(\mathbf{R}^n)$ and hence in $C^0_{b}(\mathbf{R}^n)$ as $\varepsilon\to 0$, and $u_{0}$ is the unique positive radial ground state solution of (\ref{e:limit}).
\end{theorem}

If $s = 1$, (\ref{e:main}) is the classical Schr\"{o}dinger equation and the corresponding result of Theorem \ref{t:main} was established by Floer and Weinstein in \cite{Floer86}. There are large mounts of research on this equation in the past two decades. We limit ourselves to citing a
few recent papers \cite{AMS01, AmbrosettiMalchiodiNi2003, BahriLi1990, BahriLions1997, DelpinoFelmer1996, ByeonWangARMA2002, Grossi2002, LiYanyan1997, DelpinoWeijunchengCPAM}, referring to their bibliography for a broader
list of works, although still not exhaustive.

In the case $s\in (0, 1)$, the operator $(-\Delta)^s$ on $\mathbf{R}^{n}$ is nonlocal, while $-\Delta$ is local. As pointed out in \cite{Secchi2013}, when studying the singularly perturbed equation (\ref{e:main}), the standard techniques that were developed for the local laplacian do not work out-of-the-box since these techniques heavily rely on blow-up and local estimates, and need fine properties of solutions to the limiting problem. Moreover, $(-\Delta)^s$ may kill bumps by averaging on the whole space. Fortunately, based on the work of \cite{Frank&LenzmannActaMath}, \cite{FrankLenzmann&Silvestre} and by carefully using the cut-off function technique, we can recover the main ingredients of the Lyapunov-Schmidt reduction method in the fractional case (although the ground bound state got in \cite{FrankLenzmann&Silvestre} not decay exponentially, the speed of decay is enough for us to obtain the estimate).

The rest of this paper is organized as follows. In Section 2, we recall the notations of fractional Laplacian and some known results, especially the uniqueness and nondegenerace results of \cite{Frank&LenzmannActaMath} and \cite{FrankLenzmann&Silvestre}. In Section 3, we prove the invertibility of the linearized operator at the ground state solution. In Section 4 and 5, we prove the main result of this paper by the Lyapunov-Schmidt reduction method.

\section{Preliminaries}
In this section, we recall some properties of the fractional order Sobolev spaces and the results of \cite{Frank&LenzmannActaMath}, \cite{FrankLenzmann&Silvestre} which are crucial in our proof of the main theorem.
\subsection{Fractional order Sobolev spaces}
In this subsection, we recall some useful facts of the fractional order Sobolev spaces. For more details, please see, for example, \cite{AdamsSobolevSpace}, \cite{Shubin01}, \cite{NezzaPalatucci&Valdinoci}.

Consider the Schwartz space $\mathcal{S}$ of rapidly decaying $C^{\infty}$ functions on $\mathbf{R}^{n}$. The topology of this space is generated by the seminorms
$$
p_{N}(\varphi) = \displaystyle\sup_{x\in\mathbf{R}^{n}}(1 + |x|)^{N}\sum_{|\alpha|\leq N}|D^{\alpha}\varphi(x)|,\quad N = 0, 1, 2,\cdot\cdot\cdot,
$$
where $\varphi\in \mathcal{S}$. Let $\mathcal{S}'$ be the set of all tempered distributions, which is the topological dual of $\mathcal{S}$. As usual, for any $\varphi\in \mathcal{S}$, we denote by
$$
\mathcal{F}\varphi(\xi) = \frac{1}{(2\pi)^{n/2}}\int_{\mathbf{R}^{n}}e^{-i\xi\cdot x}\varphi(x)dx
$$
the Fourier transformation of $\varphi$ and we recall that one can extend $\mathcal{F}$ from $\mathcal{S}$ to $\mathcal{S}'$.

When $s\in (0, 1)$, the space $H^{s}(\mathbf{R}^{n}) = W^{s, 2}(\mathbf{R}^n)$ is defined by
\begin{eqnarray*}
H^{s}(\mathbf{R}^{n})& = &\left\{u\in L^2(\mathbf{R}^2): \frac{|u(x) - u(y)|}{|x - y|^{\frac{n}{2} + s}}\in L^{2}(\mathbf{R}^n\times\mathbf{R}^n)\right\}\\
& = & \left\{u\in L^2(\mathbf{R}^2): \int_{\mathbf{R}^n}(1 + |\xi|^{2s})|\mathcal{F}u(\xi)|^2d\xi < +\infty\right\}
\end{eqnarray*}
and the norm is
\begin{eqnarray*}
\|u\|_{s} := \|u\|_{H^{s}(\mathbf{R}^{n})}& = &\left(\int_{\mathbf{R}^n}|u|^2dx + \int_{\mathbf{R}^n}\int_{\mathbf{R}^n}\frac{|u(x) - u(y)|^2}{|x - y|^{n + 2s}}dxdy\right)^\frac{1}{2}.
\end{eqnarray*}
Here the term
$$
[u]_{s} := [u]_{H^{s}(\mathbf{R}^{n})} = \left(\int_{\mathbf{R}^n}\int_{\mathbf{R}^n}\frac{|u(x) - u(y)|^2}{|x - y|^{n + 2s}}dxdy\right)^\frac{1}{2}
$$
is the so-called Gagliardo (semi) norm of $u$. The following identity yields the relation between the fractional operator $(-\Delta)^s$ and the fractional Sobolev space $H^{s}(\mathbf{R}^{n})$,
$$
[u]_{H^{s}(\mathbf{R}^{n})} = C\left(\int_{\mathbf{R}^n}|\xi|^{2s}|\mathcal{F}u(\xi)|^2d\xi\right)^{\frac{1}{2}} = C\|(-\Delta)^{\frac{s}{2}}u\|_{L^2(\mathbf{R}^n)}
$$
for a suitable positive constant $C$ depending only on $s$ and $n$.

When $s > 1$ and it is not an integer we write $s = m + \sigma$, where $m$ is an integer and $\sigma\in (0, 1)$. In this case the space $H^{s}(\mathbf{R}^{n})$
consists of those equivalence classes of functions $u\in H^{m}(\mathbf{R}^{n})$ whose distributional derivatives $D^{J} u$, with $|J| = m$, belong to $H^{\sigma}(\mathbf{R}^{n})$, namely
\begin{eqnarray*}
H^{s}(\mathbf{R}^{n}) = \left\{u\in H^{m}(\mathbf{R}^{n}): D^{J} u\in H^{\sigma}(\mathbf{R}^{n}) \mbox{\,\,for any\,\,}J \mbox{\,\,with\,\,} |J| = m\right\}
\end{eqnarray*}
and this is a Banach space with respect to the norm
\begin{eqnarray*}
\|u\|_{s} := \|u\|_{H^{s}(\mathbf{R}^{n})} = \left(\|u\|^2_{H^{m}(\mathbf{R}^{n})} + \displaystyle\sum_{|J| = m}\|D^{J} u\|^2_{H^{\sigma}(\mathbf{R}^{n})}\right)^\frac{1}{2}.
\end{eqnarray*}
Clearly, if $s = m$ is an integer, the space $H^{s}(\mathbf{R}^{n})$ coincides with the usual Sobolev space $H^{m}(\mathbf{R}^{n})$.

For a general domain $\Omega$, the space $H^s(\Omega)$ can be defined similarly.

On the Sobolev inequality and the compactness of embedding, one has
\begin{theorem}\cite{AdamsSobolevSpace}\label{l:embedding}
Let $\Omega$ be a domain with smooth boundary in $\mathbf{R}^n$. Let $s > 0$, then
\begin{enumerate}
\item[(a)]If $n > 2s$, then $H^{s}(\Omega)\to L^{r}(\Omega)$ for $2\leq r \leq 2n/(n - 2s)$,
\item[(b)]If $n = 2s$, then $H^{s}(\Omega)\to L^{r}(\Omega)$ for $2\leq r < \infty$,
\item[(c)]If $n < 2(s - j)$ form some nonnegative integer $j$, then $H^{s}(\Omega)\to C_{b}^j(\Omega)$.
\end{enumerate}
\end{theorem}
\begin{theorem}\cite{Shubin01}\label{l:compactness}
Let $s>s'$ and $\Omega$ be a bounded domain with smooth boundary in $\mathbf R^n$. Then the embedding operator
\begin{equation*}
i_s^{s'}:H^s(\Omega)\to H^{s'}(\Omega)
\end{equation*}
is compact.
\end{theorem}

As the usual singularly perturbed problem, when dealing with (\ref{e:main}), we rescale the variable $x$ so that the term involving $V(x)$ appears as a small perturbation. Without loss of generality, we assume that
\begin{center}
{\bf the non-degenerate minimum point of $V$ lies at the origin with value $1$.}
\end{center}

Let $y=x/\varepsilon$ and $u(y)=v(\varepsilon y)$. Then $u$ satisfies
\begin{equation}
(-\Delta)^su+(V_{\varepsilon}-1)u+u-|u|^{\alpha}u=0,
\end{equation}
where $V_{\varepsilon}=V(\varepsilon y)$. Note that $V_{\varepsilon}$ approaches $V(0) = 1$ uniformly on any compact set as $\varepsilon \to 0$. Set
\begin{equation}
S_{\varepsilon}(u)=(-\Delta)^su+(V_{\varepsilon}-1)u+u-|u|^{\alpha}u.
\end{equation}
Therefore, as $\varepsilon\to 0$, $S_{\varepsilon}$ has the limit
\begin{equation}
S_0(u)=(-\Delta)^su+u-|u|^{\alpha}u.
\end{equation}

\subsection{Uniqueness and non-degeneracy for the limit equation}
In this subsection, we recall some known results for the limit equation $S_0(u)=0$, i.e., (\ref{e:limit}).

If $s = 1$, the Uniqueness and non-degeneracy of the ground state for (\ref{e:limit}) is due to \cite{KwongARMA1989}.

In the celebrated paper \cite{Frank&LenzmannActaMath}, Frank and Lenzmann proved uniqueness of ground state solution $u_0 = u_0(|x|)\geq 0$ for (\ref{e:limit}) where $n = 1$, $0 < s < 1$, $0 < \alpha < \alpha_{*}(s, 1)$. This result generalized the specific uniqueness result obtained by Amick and Toland for $s = 1/2$ and $\alpha = 1$ in \cite{Amick&Toland91}. They also showed that the associated linearized operator $L_{0} = (-\Delta)^s+1-(\alpha+1)u_{0}^{\alpha}$ is nondegenerate, i.e., its kernel satisfies $\text{ker} L_{0} = \text{span}\{u_{0}'\}$.  This results plays a central role for the stability of solitary waves and blow up analysis for nonlinear dispersive PDEs with fractional Laplacians, such as the generalized Benjamin-Ono and Benjamin-Bona-Mahony water wave equations. We recall that $u_{0}\geq 0$ with $u_{0}\not\equiv 0$ is a ground state solution of (\ref{e:limit}), if $L_{0}$ has Morse index one, i.e., $L_{0}$ has exactly one strictly negative eigenvalue (counting multiplicity), see \cite{ChangKungching93}.

Very recently, in the paper \cite{FrankLenzmann&Silvestre}, Frank, Lenzmann and Silvestre proved uniqueness and nondegeneracy of ground state solutions for (\ref{e:limit}) in arbitrary dimension $n \geq 1$ and any admissible exponent $0 < \alpha < \alpha_{*}(s, n)$. This result classifies all optimizers of the fractional Gagliardo-Nirenberg-Sobolev inequality
\begin{equation}\label{e£ºGNSinequality}
\int_{\mathbf R^n}|u|^{\alpha+2}dx\le C\left(\int_{\mathbf R^n}|(-\Delta)^{s/2}u|^2dx\right)^{\frac{n\alpha}{4s}}\left(\int_{\mathbf R^n}|u|^2dx\right)^{\frac{\alpha+2}{2}-\frac{n\alpha}{4s}}.
\end{equation}

For convenience, we summarize the results of \cite{Frank&LenzmannActaMath} and \cite{FrankLenzmann&Silvestre} in the following theorems.
\begin{theorem}\label{l:decay}
Let $n\geq 1$, $s\in (0, 1)$, and $0 < \alpha < \alpha_{*}(s, n)$. Then the following holds.
\begin{enumerate}
\item[(i)]{\bf (Existence)}: There exists a function $0\leq u_{0}\in H^{s}(\mathbf{R}^n)$ which solves equation (\ref{e:limit}).
\item[(ii)]{\bf (Symmetry, regularity and decay)}: If $u_{0}\in H^{s}(\mathbf{R}^n)$ with $u_{0}\geq 0$ and $u_{0}\not\equiv 0$ solves (\ref{e:limit}), then there exists some $x_{0}\in \mathbf{R}^n$ such that $u_{0}(\cdot - x_{0})$ is radial, positive and strictly decreasing in $|x-x_{0}|$. Moreover, the function $u_{0}$ belongs to $H^{2s+1}(\mathbf{R}^n)\cap C^{\infty}(\mathbf{R}^n)$ and it satisfies
    \begin{equation}
    \frac{C_{1}}{1 + |x|^{n + 2s}} \leq u_{0}(x) \leq \frac{C_{2}}{1 + |x|^{n + 2s}}\quad\mbox{for}\quad x\in\mathbf{R}^n,
    \end{equation}
    with some constants $C_{2}\geq C_{1} > 0$.
\end{enumerate}
\end{theorem}
\begin{theorem}
{\bf (Nondegeneracy)} Let $n\geq 1$, $s\in(0, 1)$, and $0 < \alpha < \alpha_{*}(s, n)$. Suppose that $0\leq u_{0}\in H^{s}(\mathbf{R}^n)$ is a ground state solution of (\ref{e:limit}). Then
the linearized operator $L_{0} = (-\Delta)^s + 1 - (\alpha + 1)|u_{0}|^{\alpha}$ is nondegenerate, i.e., its kernel is given by
$$
{\rm ker} L_{0} = {\rm span}\{\partial_{x_{1}}u_{0},\cdot\cdot\cdot, \partial_{x_{n}}u_{0}\}.
$$
\end{theorem}
\begin{theorem}
{\bf (Uniqueness)} Let $n\geq 1$, $s\in(0, 1)$, and $0 < \alpha < \alpha_{*}(s, n)$. The ground state solution $u_{0}\in H^{s}(\mathbf{R}^n)$ for equation (\ref{e:limit}) is unique up to translation.
\end{theorem}
The nondegeneracy implies that $0$ is an isolated spectral point of $L_0$. More precisely, for all $\phi\in (\rm{ker} L_0)^{\perp}$, one has
\begin{equation}
\|L_0\phi\|_{L^2(\mathbf{R}^n)}\ge c\|\phi\|_{H^{2s}(\mathbf{R}^n)}
\end{equation}
for some positive constant $c$. By Lemma C.2 of \cite{FrankLenzmann&Silvestre}, it holds that, for $j = 1,\cdot\cdot\cdot, n$, $\partial_j u_0 : = \partial_{x_j} u_0$ has the following decay estimate,
\begin{equation}
|\partial_j u_0|\leq \frac{C}{1 + |x|^{n + 2s}}.
\end{equation}

It is well known that when $s = 1$, the ground state solution of (\ref{e:limit}) decays exponentially at infinity. But from Thoerem \ref{l:decay}, when $s\in(0, 1)$, the corresponding ground bound state solution decays like $\frac{1}{|x|^{n + 2s}}$ when $|x|\to \infty$. Fortunately, this polynomial decay is enough for us in the estimates of our proof, see Section 3, 4 and 5.

\section{Linearized operator at the ground state solution}
In this section, we study the linearized operator at the ground state solution $u_{0}$.
\subsection{Linearized operator.}
Denote by $\|\cdot\|_0$ and $\|\cdot\|_{2s}$ the norm in $L^2=L^2(\mathbf R^n)$ and $H^{2s}=H^{2s}(\mathbf R^n)$ respectively. We have the following lemma.
\begin{lemma}
For any $\varepsilon>0$, $S_{\varepsilon}$ has continuous Fr\'{e}chet derivative
\begin{equation}
S'_{\varepsilon}(u)=(-\Delta)^s+V_{\varepsilon}-(\alpha+1)|u|^{\alpha}, \quad u\in H^{2s}.
\end{equation}
\end{lemma}
\begin{proof}
Let $u$ be any point in $H^{2s}$. The continuity for $(-\Delta)^s+V_{\varepsilon}$ is obvious, so we only need to prove that for any $w\in H^{2s}$, $|u|^{\alpha}w\in L^2$. In fact,
by Theorem \ref{l:embedding}, we have imbedding
$$H^{2s}(\mathbf R^n)\hookrightarrow L^{q}(\mathbf R^n),$$ $ q\in[2,\frac{2n}{n-4s}]$ for $4s< n$, and $q\in[2,\infty)$ for $4s\ge n$.
Since $$2\alpha+2\in (2,2\alpha_*(s,n)+2)=(2,\frac{2n+4s}{n-2s}),$$
$2\alpha+2$ is in $[2,\frac{2n}{n-4s}]$ for $4s< n$, and in $[2,\infty)$ for $4s\ge n$.
Then by the H\"{o}lder inequality,
\begin{eqnarray*}
\displaystyle\int_{\mathbf{R}^n} (|u|^{\alpha}|w|)^2 dy &\le& \left(\displaystyle\int_{\mathbf{R}^n} |u|^{2\alpha\cdot\frac{2\alpha+2}{2\alpha}}dy\right)^{\frac{\alpha}{\alpha+1}}\left(\displaystyle\int_{\mathbf{R}^n} |w|^{2\cdot\frac{2\alpha+2}{2}}dy\right)^{\frac{1}{\alpha+1}}\\
&\le&\left(\displaystyle\int_{\mathbf{R}^n} |u|^{2\alpha+2}dy\right)^{\frac{\alpha}{\alpha+1}}\left(\displaystyle\int_{\mathbf{R}^n} |w|^{2\alpha+2}dy\right)^{\frac{1}{\alpha+1}}\\
&\leq& C\|u\|_{2s}^{2\alpha}\|w\|_{2s}^2.
\end{eqnarray*}
This completes the proof.
\end{proof}

Let the approximate solution be
\begin{equation}
u_{z,\varepsilon}(y)=u_0\left(y-\frac{z}{\varepsilon}\right).
\end{equation}
The Taylor's expansion of $S_{\varepsilon}$ at $u_{z,\varepsilon}$ is
\begin{equation}\label{e:expansion}
S_{\varepsilon}(u_{z,\varepsilon}+\phi)=S_{\varepsilon}(u_{z,\varepsilon})+S'_{\varepsilon}(u_{z,\varepsilon})\phi+N_{z,\varepsilon}(\phi),
\end{equation}
where
\begin{equation}\label{e:ne}
N_{z,\varepsilon}(\phi)=-\left(|u_{z,\varepsilon}+\phi|^{\alpha}(u_{z,\varepsilon}+\phi)
-|u_{z,\varepsilon}|^{\alpha}u_{z,\varepsilon}-(\alpha+1)|u_{z,\varepsilon}|^{\alpha}\phi\right).
\end{equation}

\subsection{Invertibility of the orthogonal projection of $S_{\varepsilon}'(u_{z,\varepsilon})$.}
We want to invert $S'_{\varepsilon}(u_{z,\varepsilon})$ and obtain a fixed point equation for $\phi$ by setting expansion (\ref{e:expansion}) to be zero.
Let $K_{z,\varepsilon}$ be the kernel of $S'_0(u_{z,\varepsilon})$, which is spanned by $\{\partial_{1}u_{z,\varepsilon},\cdot\cdot\cdot, \partial_{n}u_{z,\varepsilon}\}$. Let
\begin{equation}
\pi^{\perp}_{z,\varepsilon}:L^2\to K^{\perp}_{z,\varepsilon},
\end{equation}
and
\begin{equation}
L_{z,\varepsilon}=\pi^{\perp}_{z,\varepsilon} S'_{\varepsilon}(u_{z,\varepsilon})|_{K^{\perp}_{z,\varepsilon}\cap H^{2s}}.
\end{equation}
Using the fact that $S_{\varepsilon}$ is near $S_0$ for sufficiently small $\varepsilon$, we have
\begin{lemma}\label{l:invertible}
There exist positive numbers $\beta$, $r_1$, $\varepsilon_1$ such that for $|z|<r_1$, $0<\varepsilon<\varepsilon_1$, and $\phi\in K^{\perp}_{z,\varepsilon}\cap H^{2s}$, it holds
\begin{equation}
\|L_{z,\varepsilon}\phi\|_{0}\ge \beta\|\phi\|_{2s}.
\end{equation}
\end{lemma}
\begin{proof}
If this lemma is wrong, then there exists a sequence of $(z_i,\varepsilon_i)\to 0$ in $\mathbf R^n\times \mathbf R^+$ and a sequence $\phi_i\in K^{\perp}_{z,\varepsilon}\cap H^{2s}$ satisfying
\begin{equation}\label{e:psi1}
\|\phi_i\|_{2s}=1
\end{equation}
and
\begin{equation}\label{e:lpl}
\|L_{z_i,\varepsilon_i}\phi_i\|_0\to 0, \quad \text{as } i\to \infty.
\end{equation}
Let
\begin{equation}\label{e:psit}
\psi_i(y)=\phi_i(y+\frac{z_i}{\varepsilon_i}).
\end{equation}
Since $\|\psi_i\|_{2s}=\|\phi_i\|_{2s}=1$ for all $i$, we may assume (passing to a subsequence) that $\psi_i$ converge weakly to some $\psi_{\infty}$ in $H^{2s}$. Moreover, $\psi_i$ is also $L^2$-orthogonal to $K_{0} := \rm{ker} S'_0(u_{0}) = \rm{ker} L_{0}$ (recall that $L_0 = S'_0(u_0)$), so it follows that the same is true for the weak limit $\psi_{\infty}$, i.e., $\psi_{\infty}\in K_{0}^\perp$.

{\bf Claim 1:} \emph{We have $L_0\psi_{\infty}=0$, which implies that $\psi_{\infty}=0$}.

In fact, define the linear operators
\begin{equation}\label{e:lipi}
L_i=\pi_0^{\perp}((-\Delta)^s+V_i-(\alpha+1)|u_0|^{\alpha}),
\end{equation}
where
\begin{equation}
V_i(y)=V_{\varepsilon_i}(y+\frac{z_i}{\varepsilon_i})=V(\varepsilon_i y+z_i)
\end{equation}
and $\pi_0^{\perp}$ is the orthogonal projection onto the complement of ${\rm ker} L_0$ in $L^2$. Obviously we have
\begin{equation}\label{e:lpp}
L_i\psi_i(y)=(L_{z_i,\varepsilon_i}\phi_i)(y+\frac{z_i}{\varepsilon_i}).
\end{equation}
Let $\Omega\subset \mathbf R^n$ be any bounded domain with smooth boundary. Defining $\|f\|_{0,\Omega}=\left(\int_{\Omega}|f(y)|^2dy\right)^{\frac{1}{2}}$, we have
\begin{eqnarray*}
\|L_0\psi_i\|_{0,\Omega}&=&\|\pi_0^{\perp}L_0\psi_i\|_{0,\Omega}\\
&=&\|(L_i-\pi_0^{\perp}(V_i-1))\psi_i\|_{0,\Omega}\\
&\le& \|L_i\psi_i\|_{0}+(\max_{y\in\Omega}|V_i(y)-1|) \|\psi_i\|_{0}.
\end{eqnarray*}
Since $V\in C_b^3(\mathbf R^n)$ and $ \|\psi_i\|_{0}=1$, we have that $(\max_{y\in\Omega}|V_i(y)-1|) \|\psi_i\|_{0}\to 0$. By (\ref{e:lpl}) and (\ref{e:lpp}), we obtain that
\begin{equation}\label{e:l0lp}
\lim_{i\to\infty}\|L_0\psi_i\|_{0,\Omega}=0.
\end{equation}
Because $\psi_i\rightharpoonup \psi_{\infty}$ in $H^{2s}$, we have that $L_0\psi_i\rightharpoonup L_0\psi_{\infty}$ in $L^2$. Then $L_0\psi_i\rightharpoonup L_0\psi_{\infty}$ in $L^2(\Omega)$. From (\ref{e:l0lp}), we get that $L_0\psi_{\infty}=0$. Since $\psi_{\infty}\in {\rm ker}L_0^\perp$, it follows that $\psi_{\infty}=0$. This completes the proof of Claim 1.

{\bf Claim 2:} \emph{For any bounded domain $\Omega\subset \mathbf R^n$ with smooth boundary, $\psi_i\to 0$ in $L^2(\Omega)$ as $i\to \infty$.}

In fact, let
$$\Omega_{r}=\{x\in \mathbf R^n\,|\,|x-x_0|\le r, \text{ for some point $x_0\in \Omega$}\},$$
and $\eta$ is a smooth cut-off function such that
\begin{equation}
\eta(x)=\begin{cases} 1 , &  \text{ for }x \in \Omega,\\
0, & \text{ for $\mathbf R^n\setminus\Omega_{\frac{1}{2}}$}.
\end{cases}
\end{equation}
Then $\{\eta \psi_i\}$ be a bounded sequence of function in $H^{2s}(\Omega_{1})$. Since $H^{2s}(\Omega_1)$ is a Hilbert space, there exists a function $f\in H^{2s}(\Omega_1)$ such that
$$\{\eta \psi_i\}\rightharpoonup f, \quad\text{(up to a subsequence).}$$
By the compactness of embedding $i_{2s}^0:H^{2s}(\Omega_1)\to L^2(\Omega_1)$ (Theorem \ref{l:compactness}), we have that
$$\eta\psi_i\to f \quad \text{in $L^2(\Omega_1)$}.$$
Then
$$\eta\psi_i|_{\Omega}\to f|_{\Omega} \quad \text{in $L^2(\Omega)$},$$
that is,
$$\psi_i|_{\Omega}\to f|_{\Omega}\quad \text{in $L^2(\Omega)$}.$$
However, by the argument in Claim 1, $\psi_i$ weakly converges to $0$ in $L^2(\Omega)$. Therefore,
\begin{equation}
\psi_i\to 0\quad \text{in $L^2(\Omega)$}.
\end{equation}
So we have Claim 2.

With these two claims, we now prove the lemma. Since $(\alpha+1)|u_0|^{\alpha}$ decays at infinity, it follows that
\begin{equation}\label{e:apl}
(\alpha+1)|u_0|^{\alpha}\psi_i\to 0\quad\text{in $L^2(\mathbf R^n)$.}
\end{equation}
 From (\ref{e:lipi}), (\ref{e:lpp}), (\ref{e:lpl}) and (\ref{e:apl}), we obtain that
 \begin{equation}\label{e:bperp}
 \lim_{i\to \infty}\|\pi_0^{\perp}G_i\psi_i\|_0=0,
 \end{equation}
where $G_i=(-\Delta)^s+V_i$.

Since $G_i$ is self-adjoint, we have that
\begin{equation}\label{e:bselfadjoint}
\langle\partial_ju_0,G_i\psi_i\rangle_{L^2(\mathbf R^n)}=\langle G_i(\partial_ju_0),\psi_i\rangle_{L^2(\mathbf R^n)}.
\end{equation}
Since
\begin{equation}
S'_0(u_0)(\partial_ju_0)=0,\quad \text{for }1\le j\le n,
\end{equation}
 we have
 \begin{equation}
 G_i(\partial_ju_0)=(V_i-1+(\alpha+1)|u_0|^{\alpha})(\partial_ju_0), \quad 1\le j\le n.
 \end{equation}
By (\ref{e:bselfadjoint}),
\begin{equation}\label{e:bv_expansion}
\langle\partial_ju_0,G_i\psi_i\rangle_{L^2(\mathbf R^n)}=\langle(V_i-1)(\partial_ju_0),\psi_i\rangle_{L^2(\mathbf R^n)}
+\langle(\alpha+1)|u_0|^{\alpha}\partial_ju_0,\psi_i\rangle_{L^2(\mathbf R^n)}.
\end{equation}
Estimating the first term on the right hand side of (\ref{e:bv_expansion}), we obtain
\begin{eqnarray}
&&\left|\langle(V_i-1)(\partial_ju_0),\psi_i\rangle_{L^2(\mathbf R^n)}\right|\\
&&\le\int_{\mathbf R^n}|(V_i-1)(\partial_ju_0)|\,|\psi_i|dy\notag\\
&&\le\int_{B_\rho}|(V_i-1)(\partial_ju_0)|\,|\psi_i|dy+\int_{\mathbf R^n\setminus B_\rho}|(V_i-1)(\partial_ju_0)|\,|\psi_i|dy,
\end{eqnarray}
where $\rho>0$ and $B_{\rho}$ is the ball centered at $0$ with radius $\rho$ in $\mathbf R^n$. By the H\"{o}lder inequality,
\begin{eqnarray}
&&\int_{B_\rho}|(V_i-1)(\partial_ju_0)|\,|\psi_i|dy\\
&&\le\left(\int_{B_\rho}|(V_i-1)(\partial_ju_0)|^2dy\right)^{\frac{1}{2}}\left(\int_{B_\rho}|\psi_i|^2dy\right)^{\frac{1}{2}}.\notag
\end{eqnarray}
Since $(V_i-1)\partial_ju_0\to 0$ uniformly in $B_{\rho}$ for some fixed $\rho>0$ and $\|\psi_i\|_0$ is bounded, we have
\begin{equation}
\int_{B_\rho}|(V_i-1)(\partial_ju_0)|\,|\psi_i|dy\to 0.
\end{equation}
Again by the H\"{o}lder inequality and $V\in C_b^3$,
\begin{eqnarray}
&&\int_{\mathbf R^n\setminus B_\rho}|(V_i-1)(\partial_ju_0)|\,|\psi_i|dy\\
&&\le\left(\int_{\mathbf R^n\setminus B_\rho}|(V_i-1)(\partial_ju_0)|^2dy\right)^{\frac{1}{2}}\left(\int_{\mathbf R^n\setminus B_\rho}|\psi_i|^2dy\right)^{\frac{1}{2}}\notag\\
&&\le C\left(\int_{\mathbf R^n\setminus B_\rho}|\partial_ju_0|^2dy\right)^{\frac{1}{2}}\left(\int_{\mathbf R^n\setminus B_\rho}|\psi_i|^2dy\right)^{\frac{1}{2}}\notag.
\end{eqnarray}
Since $\partial_ju_0\in L^{2}(\mathbf R^n)$ for $1\le j\le n$ and $\psi_i$ is bounded in $L^{2}(\mathbf R^n)$, we have
\begin{equation*}
\int_{\mathbf R^n\setminus B_\rho}|(V_i-1)(\partial_ju_0)|\,|\psi_i|dy
\end{equation*}
is sufficiently small if $\rho$ is large enough.
Therefore $\left|\langle(V_i-1)(\partial_ju_0),\psi_i\rangle_{L^2(\mathbf R^n)}\right|$ goes to $0$ as $i\to\infty$.
The second term (\ref{e:bv_expansion}) goes to $0$ since $\psi_i$ converges weakly to $0$. Therefore, we have
\begin{equation}\label{e:bker}
\langle\partial_ju_0,G_i\psi_i\rangle_{L^2(\mathbf R^n)}\to 0,\quad 1\leq j\leq n.
\end{equation}

Hence, by (\ref{e:bperp}) and (\ref{e:bker}), we have
\begin{equation}\label{e:convergence}
G_i\psi_i\to 0 \mbox{\,\,in\,\,} L^2.
\end{equation}

On the other hand, for any $\varphi\in H^{2s}(\mathbf{R}^n)$,
\begin{eqnarray}
&&\langle\varphi,G_i\varphi\rangle_{L^2}\ge \langle\varphi,((-\Delta)^s+1)\varphi\rangle_{L^2}\ge \langle\varphi,\varphi\rangle_{L^2}.
\end{eqnarray}
So for all $1\le i< \infty$, the spectrum of $G_i$ is contained in $[1,\infty)$. Therefore, $G_i$ is invertible, and the operator norm of $G_i^{-1}$ from $L^2(\mathbf{R}^n)$ to $H^{2s}(\mathbf{R}^n)$ is not greater then 1. Thus, we obtain
\begin{equation}
1=\|\psi_i\|_{2s}=\|G_i^{-1}G_i \psi_i\|_{2s}\le \|G_i\psi_i\|_{0}.
\end{equation}
By (\ref{e:convergence}), this is impossible.

This completes the proof.
\end{proof} 
\section{Nonlinear problem}
In this section, we shall prove that for each sufficiently small $z$ and $\varepsilon$, there is an element $\phi_{z,\varepsilon}$ in $K^{\perp}_{z,\varepsilon}$ such that
\begin{equation}\label{e:fixed}
\pi_{z,\varepsilon}^{\perp}S_{\varepsilon}(u_{z,\varepsilon}+\phi_{z,\varepsilon})=0.
\end{equation}
From now on, we assume $s > \max\{\frac{n}{4}, \frac{1}{2}\}$.

By the expansion (\ref{e:expansion}), we have
\begin{equation}
\pi_{z,\varepsilon}^{\perp}S_{\varepsilon}(u_{z,\varepsilon})+\pi_{z,\varepsilon}^{\perp}S'_{\varepsilon}(u_{z,\varepsilon})\phi
+\pi_{z,\varepsilon}^{\perp}N_{z,\varepsilon}(\phi)=0.
\end{equation}
For simplicity, denote $\pi_{z,\varepsilon}^{\perp}S_{\varepsilon}(u_{z,\varepsilon})$ by $S_{z,\varepsilon}^{\perp}$, and $\pi_{z,\varepsilon}^{\perp}N_{z,\varepsilon}$ by $N_{z,\varepsilon}^{\perp}$. From Lemma \ref{l:invertible}, we know that $L_{z, \varepsilon}$ is invertible. Then Equation (\ref{e:fixed}) is equivalent to a fixed point of the map $M_{z,\varepsilon}$ on $H^{2s}(\mathbf R^n)$ given by
\begin{equation}
M_{z,\varepsilon}(\phi)=-L_{z,\varepsilon}^{-1}(N_{z,\varepsilon}^{\perp}(\phi)+S_{z,\varepsilon}^{\perp}).
\end{equation}
We will prove that $M_{z,\varepsilon}$ is a contraction on a suitable neighborhood of $0$.
\begin{lemma}\label{l:contracting-estimate}
There exist positive constants $C$, $\delta$ independent of $z$ and $\varepsilon$, such that for all $\phi_1$, $\phi_2\in H^{2s}$ with $\|\phi_1\|_{2s}\le \delta$, $\|\phi_2\|_{2s}\le\delta$, it holds that
\begin{equation}\label{e:nonlinear-estimate}
\|N_{z,\varepsilon}(\phi_1)\|_0\le C\|\phi_1\|_{2s}^2
\end{equation}
and
\begin{equation}
\|N_{z,\varepsilon}(\phi_2)-N_{z,\varepsilon}(\phi_1)\|_0\le C\max (\|\phi_1\|_{2s},\|\phi_2\|_{2s})\|\phi_2-\phi_1\|_{2s}.
\end{equation}
\end{lemma}
\begin{proof}
Since by assumption $s > \frac{n}{4}$, it follows from Theorem \ref{l:embedding} that
\begin{equation}
H^{2s}(\mathbf R^n)\hookrightarrow L^4(\mathbf R^n)
\end{equation}
and
\begin{equation}\label{e:imbedding-C}
H^{2s}(\mathbf R^n)\hookrightarrow C_b^0(\mathbf R^n).
\end{equation}

By (\ref{e:ne}) and the imbedding above, we have, for $\|\phi\|_{2s}\le 1$,
\begin{eqnarray*}
\int_{\mathbf R^n}|N_{z,\varepsilon}(\phi)|^2dx
&=&\int_{\mathbf R^n}\left||u_{z,\varepsilon}+\phi|^{\alpha}(u_{z,\varepsilon}+\phi)-|u_{z,\varepsilon}|^{\alpha}u_{z,\varepsilon}-(\alpha+1)|u_{z,\varepsilon}|^{\alpha}\phi\right|^2dx\\
&\le&C\int_{\mathbf R^n}|u_{z,\varepsilon}+\theta_1\phi|^{2(\alpha-1)}|\phi|^4 dx\le C\|\phi\|_{L^4}^4\le c\|\phi\|_{2s}^4,
\end{eqnarray*}
where $\theta_1$ is a positive function with value not greater than $1$.
For the second inequality, we compute
\begin{eqnarray*}
&&\int_{\mathbf R^n}|N_{z,\varepsilon}(\phi_1)-N_{z,\varepsilon}(\phi_2)|^2dx\\
&=&\int_{\mathbf R^n}\left||u_{z,\varepsilon}+\phi_1|^{\alpha}(u_{z,\varepsilon}+\phi_1)-|u_{z,\varepsilon}+\phi_2|^{\alpha}(u_{z,\varepsilon}+\phi_2)-(\alpha+1)|u_{z,\varepsilon}|^{\alpha}(\phi_1-\phi_2)\right|^2dx\\
&\le&C\int_{\mathbf R^n}|(u_{z,\varepsilon}+\theta_3(\phi_2+\theta_2(\phi_1-\phi_2)))|^{2(\alpha-1)}|(\phi_2+\theta_2(\phi_1-\phi_2))|^2|\phi_1-\phi_2|^2dx\\
&\le&C\|\phi_2+\theta_2(\phi_1-\phi_2)\|_4^2\|\phi_1-\phi_2\|_4^2\le C\left(\max (\|\phi_1\|_{2s},\|\phi_2\|_{2s})\right)^2\|\phi_2-\phi_1\|^2_{2s}.
\end{eqnarray*}
Here $\theta_2$ and $\theta_3$ are functions with similar property as $\theta_1$.

\end{proof}

In the following, for any function $f$, we denote the maximum of $f$ on the closed ball $B_r$ of radius $r$ at $z$ by $f_r(z)$.
\begin{lemma}\label{l:contraction}
There exists a positive constant $C$ such that for every $\rho>0$,
\begin{equation}
\|S_{\varepsilon}(u_{z,\varepsilon})\|_0^2\le C(\rho^{-n-4s}+(V-1)^2_{\rho\varepsilon}(z)).
\end{equation}
Therefore, $\|S_{\varepsilon}(u_{z,\varepsilon})\|_0\to 0$ as $(z,\varepsilon)\to 0$.
\end{lemma}
\begin{proof}
Indeed, 
\begin{eqnarray*}
\|S_{\varepsilon}(u_{z,\varepsilon})\|_0^2&=&\int_{\mathbf R^n}(V_{\varepsilon}(y)-1)^2 u_{z,\varepsilon}^2(y)dy\\
&=& \int_{\mathbf R^n}(V_{\varepsilon}(y+\frac{z}{\varepsilon})-V(0))^2u_0^2(y)dy\\
&=& \int_{B_{\rho}}(V(z+\varepsilon y)-V(0))^2u_0^2(y)dy+\max_{x\in \mathbf R^n}|V(x)-V(0)|^2\int_{\mathbf R^n\setminus B_{\rho}}u_0^2(y)dy\\
&\le& (V-V(0))^2_{\rho\varepsilon}(z)\|u_0\|_0^2+\max_{x\in \mathbf R^n}|V(x)-V(0)|^2\int_{\mathbf R^n\setminus B_{\rho}}\frac{1}{(1+|y|^{n+2s})^2}dy\\
&\le&C((V-V(0))^2_{\rho\varepsilon}(z)+\rho^{-n-4s}).
\end{eqnarray*}
This completes the proof.
\end{proof}

\begin{lemma}\label{l:contraction2}
There exist positive constants $\Theta$, $r_0$, $\varepsilon_0$ such that for every $z$ and $\varepsilon$ with $|z|<r_0$ and $0<\varepsilon<\varepsilon_0$, there is a unique element $\phi_{z,\varepsilon}\in H^{2s}\cap K_{z,\varepsilon}^{\perp}$ such that
\begin{equation}
S_{\varepsilon}(u_{z,\varepsilon}+\phi_{z,\varepsilon})\in K_{z,\varepsilon}
\end{equation}
and
\begin{equation}\label{e:s-control}
\|\phi_{z,\varepsilon}\|_{2s}\le \Theta\|S_{\varepsilon}(u_{z,\varepsilon})\|_0.
\end{equation}
\end{lemma}
\begin{proof}
Let $\beta$, $r_1$, $\varepsilon_1$, $C$ and $\delta$ be the constants from Lemma \ref{l:invertible} and Lemma \ref{l:contracting-estimate}. Let $\epsilon=\min(\frac{\beta}{2C},\delta)$. By Lemma \ref{l:contraction}, we choose $\varepsilon_0\le\varepsilon_1$ and $r_0\le r_1$ small enough so that
when $|z|<r_0$ and $0<\varepsilon<\varepsilon_0$,
\begin{equation}
\|S_{z,\varepsilon}^{\perp}\|_0\le \frac{1}{2}\epsilon\beta.
\end{equation}
Recall that $S_{z,\varepsilon}^{\perp}=\pi_{z,\varepsilon}^{\perp}S_{\varepsilon}(u_{z,\varepsilon})$.

{\bf Claim:} \emph{For $|z|\le r_0$ and $0< \varepsilon\le \varepsilon_0$, $M_{z,\varepsilon}$ maps the ball $B_{\epsilon}^{\perp}:=B_{\epsilon}(0)\cap K_{z,\varepsilon}^{\perp}\subset H^{2s}$ continuously into itself.}

In fact, if $\phi\in B_{\epsilon}^{\perp}$, then $M_{z,\varepsilon}(\phi)\in K_{z,\varepsilon}^{\perp}$, and
\begin{eqnarray*}
\|M_{z,\varepsilon}(\phi)\|_{2s}&\le& \frac{1}{\beta}\|N_{z,\varepsilon}(\phi)+S_{z,\varepsilon}^{\perp}\|_0\\
&\le&\frac{1}{\beta}(C\|\phi\|_{2s}^2+\|S_{z,\varepsilon}^{\perp}\|_0)\\
&\le& \frac{C\epsilon^2}{\beta}+\frac{1}{\beta}(\frac{1}{2}\epsilon\beta)\le\frac{1}{2}\epsilon+\frac{1}{2}\epsilon=\epsilon.
\end{eqnarray*}
Therefore, $M_{z,\varepsilon}(\phi)$ is also in $B_{\epsilon}^{\perp}$.
Thus we have the claim.

Next, we prove that $M_{z,\varepsilon}$ is contracting. In fact, for $\phi_1$ and $\phi_2$ in $B_{\epsilon}^{\perp}$, we have
\begin{eqnarray*}
\|M_{z,\varepsilon}(\phi_1)-M_{z,\varepsilon}(\phi_2)\|_{2s}&=& \|L_{z,\varepsilon}^{-1}(N_{z,\varepsilon}^{\perp}(\phi_1)-N_{z,\varepsilon}^{\perp}(\phi_1))\|_{2s}\\
&\le&\frac{C\epsilon}{\beta}\|\phi_1-\phi_2\|_{2s}\le \frac{1}{2}\|\phi_1-\phi_2\|_{2s}.
\end{eqnarray*}
Thus by the contraction mapping theorem, there is fixed point $\phi_{z,\varepsilon}\in B_{\epsilon}^{\perp}$ of the equation $M_{z,\varepsilon}(\phi) = \phi$. From the estimate
\begin{equation}
\|(M_{z,\varepsilon})(0)-0\|_{2s}=\|L_{z,\varepsilon}^{-1}S_{z,\varepsilon}^{\perp}\|_{2s}\le\beta^{-1}\|S_{\varepsilon}(u_{z,\varepsilon})\|_0
\end{equation}
and the fact that $M_{z,\varepsilon}$ is $\frac{1}{2}$-contracting, we obtain that
\begin{equation}
\|\phi_{z,\varepsilon}\|_{2s}\le \Theta\|S_{\varepsilon}(u_{z,\varepsilon})\|_0
\end{equation}
where $\Theta=2/\beta$.
\end{proof}

\begin{remark}
By Lemma \ref{l:contraction} and Lemma \ref{l:contraction2}, we obtain that $\|\phi_{z,\varepsilon}\|_{2s}\to 0$ as $(z,\varepsilon)\to 0$.
\end{remark}

\section{The reduced problem and proof of the main theorem}
In this section, we shall prove the main result of this paper.
\subsection{The reduced problem.} Let $r_0$ and $\varepsilon_0$ be the constants from Lemma \ref{l:contraction2}. Assume that $\varepsilon<\varepsilon_0$.
We project onto the kernel $K_{z,\varepsilon}$ to define a reduced map
\begin{eqnarray}
s_{\varepsilon} &:& B_{r_0}\to \mathbf R^n\notag\\
s_{\varepsilon}(z)&=& \frac{1}{\varepsilon}(s_{\varepsilon}(z)_1,\cdots,s_{\varepsilon}(z)_n).
\end{eqnarray}
Here $s_{\varepsilon}(z)_j=\langle S_{\varepsilon}(u_{z,\varepsilon}+\phi_{z,\varepsilon}),\partial_{y_j}u_{z,\varepsilon}\rangle$, $1\le j\le n$.
Define
\begin{equation}
v_0(z)=-\frac{1}{2}\|u_0\|^2_0 D^2V(0)z,
\end{equation}
where $D^2V(0)$ is the Hessian of $V$ at $0$,
and a family of maps $v_{\varepsilon}$ defined on $B_1$ by
\begin{equation}
v_{\varepsilon}(z)=\varepsilon^{-\nu}s_{\varepsilon}(\varepsilon^{\nu} z)
\end{equation}
 where $\nu$ is a fixed number chosen so that $\frac{1}{3}<\nu<\min(\frac{3(n+4s)-4}{(n+4s)+4},\frac{2(n+4s)-2}{(n+4s)+2})$ and $\varepsilon\le \min (\varepsilon_0,r_0^{1/\nu})$.
Then we have the following lemma.
\begin{lemma}\label{l:uniform}
As $\varepsilon\to 0$, $v_{\varepsilon}$ converge uniformly to $v_0$ on $B_1(0)$.
\end{lemma}
\begin{proof}
The expansion (\ref{e:expansion}) gives, for $1\le j\le n$,
\begin{eqnarray*}
 s_{\varepsilon}(z)_j&=&\langle \partial_{y_j}u_{z,\varepsilon},S_{\varepsilon}(u_{z,\varepsilon}+\phi_{z,\varepsilon})\rangle_{L^2}\\
 &=&\langle \partial_{y_j}u_{z,\varepsilon},S_{\varepsilon}(u_{z,\varepsilon})\rangle_{L^2}+\langle \partial_{y_j}u_{z,\varepsilon},S_{\varepsilon}'(u_{z,\varepsilon})\phi_{z,\varepsilon}\rangle_{L^2}+\langle \partial_{y_j}u_{z,\varepsilon},N_{z,\varepsilon}(\phi_{z,\varepsilon})\rangle_{L^2}\\
 &=& e_1+e_2+e_3.
\end{eqnarray*}
Since $S_0(u_{z,\varepsilon})=0$,  we have
\begin{eqnarray*}
e_1&=& \langle \partial_{y_j}u_{z,\varepsilon},(S_{\varepsilon}-S_0)(u_{z,\varepsilon})\rangle_{L^2}\\
&=& \langle \partial_{y_j}u_{z,\varepsilon},(V_{\varepsilon}-V(0))u_{z,\varepsilon}\rangle_{L^2}\\
&=& -\langle u_{z,\varepsilon},(\partial_{y_j}V_{\varepsilon}) u_{z,\varepsilon}\rangle_{L^2}-\langle u_{z,\varepsilon},(V_{\varepsilon}-V(0))(\partial_{y_j}u_{z,\varepsilon})\rangle_{L^2}.
\end{eqnarray*}
Then
\begin{eqnarray*}
e_1&=&-\frac{1}{2}\langle u_{z,\varepsilon},(\partial_{y_j}V_{\varepsilon}) u_{z,\varepsilon}\rangle_{L^2}\\
&=&-\frac{1}{2}\int_{\mathbf R^n}(\partial_{y_j}V_{\varepsilon})(y)|u_0(y-\frac{z}{\varepsilon})|^2dy\\
&=& -\frac{\varepsilon}{2}\int_{\mathbf R^n}\partial_j V(z+\varepsilon y)|u_0(y)|^2dy.
\end{eqnarray*}
So by the radial symmetry of $u_0$,
\begin{eqnarray*}
|\frac{e_1}{\varepsilon}-(v_0(z))_j|&=& \frac{1}{2}\left|\int_{\mathbf R^n}([D^2V(0)z]_j-\partial_j V(z+\varepsilon y))|u_0(y)|^2dy\right|\\
&=&\frac{1}{2}\left|\int_{\mathbf R^n}([D^2V(0)(z+\varepsilon y)]_j-\partial_j V(z+\varepsilon y))|u_0(y)|^2dy\right|\\
&\le& C\int_{\mathbf R^n} |z+\varepsilon y|^2|u_0(y)|^2dy.
\end{eqnarray*}
From the asymptotic property of $u_0$, we have, for any $\rho>0$,
\begin{eqnarray*}
&&\int_{\mathbf R^n} |z+\varepsilon y|^2|u_0(y)|^2dy\\
&&\le C\int_{\mathbf R^n}\frac{|z+\varepsilon y|^2}{(1+|y|^{n+2s})^2}dy\\
&&\le C\left(\int_{B_{\rho}}\frac{|z+\varepsilon y|^2}{(1+|y|^{n+2s})^2}dy+\int_{\mathbf R^n\backslash B_{\rho}}\frac{|z+\varepsilon y|^2}{(1+|y|^{n+2s})^2}dy\right).
\end{eqnarray*}
Then
\begin{eqnarray}
\int_{B_{\rho}}\frac{|z+\varepsilon y|^2}{(1+|y|^{n+2s})^2}dy
&\le&C(|z|+\varepsilon \rho)^2,
\end{eqnarray}
and
\begin{eqnarray*}
&&\int_{\mathbf R^n\setminus B_{\rho}}\frac{|z+\varepsilon y|^2}{(1+|y|^{n+2s})^2}dy\\
&&\le C\left(\int_{\mathbf R^n\setminus B_{\rho}}\frac{|z|^2}{(1+|y|^{n+2s})^2}dy
+\int_{\mathbf R^n\setminus B_{\rho}}\frac{|\varepsilon y|^2}{(1+|y|^{n+2s})^2}dy\right)\\
&&\le C(\rho^{-n-4s}+\varepsilon^{2}\rho^{-n-4s+2}).
\end{eqnarray*}
Therefore,
\begin{equation}
|\frac{e_1}{\varepsilon}-(v_0(z))_j|\le C((|z|+\varepsilon \rho)^2+\rho^{-n-4s}+\varepsilon^{2}\rho^{-n-4s+2}).
\end{equation}
Next, by (\ref{e:nonlinear-estimate}) and (\ref{e:s-control}), we have
\begin{equation}
|e_3|\le C\|\phi_{z,\varepsilon}\|_{2s}^2\le C\|S_{\varepsilon}(u_{z,\varepsilon})\|_0^2.
\end{equation}
Since $S_{\varepsilon}'(u_{z,\varepsilon})\partial_ju_{z,\varepsilon}=(V_{\varepsilon}-V(0))\partial_ju_{z,\varepsilon}$ and the operator $S_{\varepsilon}'(u_{z,\varepsilon})$ is self-adjoint, we obtain
\begin{equation}\label{e:v-s}
|e_2|=\langle(V_{\varepsilon}-V(0))\partial_ju_{z,\varepsilon},\phi_{z,\varepsilon}\rangle_{L^2}
\le C\|(V_{\varepsilon}-V(0))\partial_ju_{z,\varepsilon}\|_0\|S_{\varepsilon}(u_{z,\varepsilon})\|_0.
\end{equation}
Estimate by the same method in the proof of Lemma \ref{l:contraction}, we have
\begin{eqnarray*}
&&\|(V_{\varepsilon}-V(0))\partial_ju_{z,\varepsilon}\|_0^2=\int_{\mathbf R^n}|(V_{\varepsilon}-V(0))\partial_ju_{z,\varepsilon}|^2dy = \int_{\mathbf R^n}|(V_{\varepsilon}(y + \frac{z}{\varepsilon})-V(0))\partial_ju_{0}|^2dy\\
&&= \int_{B_{\rho}}|(V(\varepsilon y + z)-V(0))\partial_ju_{0}|^2dy+\int_{\mathbf R^n\setminus B_{\rho}}|(V(\varepsilon y + z)-V(0))\partial_ju_{0}|^2dy\\
&&\leq (V-V(0))^2_{\rho\varepsilon}(z)\int_{B_{\rho}}|\partial_ju_{0}|^2dy+C\int_{\mathbf R^n\setminus B_{\rho}}|\partial_ju_{0}|^2dy\\
&&\leq C\left((V-V(0))^2_{\rho\varepsilon}(z)+\int_{\mathbf R^n\setminus B_{\rho}}(\partial_ju_{0})^2dy\right)\\
&&\le C\left((V-V(0))^2_{\rho\varepsilon}(z)+\rho^{-n-4s}\right)
\end{eqnarray*}
for any $\rho>0$.

Then by Lemma \ref{l:contraction}, we have
\begin{equation}
\frac{|e_2+e_3|}{\varepsilon}\le \frac{C}{\varepsilon}(\rho^{-n-4s}+(V-1)^2_{\rho\varepsilon}(z)),\quad \text{for any $\rho>1$.}
\end{equation}
Since $0$ is a critical point of $V$, we obtain
\begin{eqnarray*}
&&|v_{\varepsilon}(z)-v_0(z)|=|\varepsilon^{-\nu}(s_{\varepsilon}(\varepsilon^{\nu} z)-v_0(\varepsilon^{\nu} z))|\\
&&\le \varepsilon^{-\nu}\left((C_1((|\varepsilon^{\nu}z|+\varepsilon \rho)^2+\rho^{-n-4s}+\varepsilon^{2}\rho^{-n-4s+2})+\frac{C}{\varepsilon}(\rho^{-n-4s}+(V-1)^2_{\rho\varepsilon}(\varepsilon^{\nu}z))\right)\\
&&\le C_1\varepsilon^{-\nu}\left((\varepsilon^{\nu}|z|+\varepsilon\rho)^2+\rho^{-n-4s}+\varepsilon^{2}\rho^{-n-4s+2}\right)
+C_3\varepsilon^{-\nu-1}\left((\varepsilon^{\nu}|z|+\varepsilon\rho)^4+\rho^{-n-4s}\right).
\end{eqnarray*}
Let $\frac{1}{3}<\nu<\min(\frac{3(n+4s)-4}{(n+4s)+4},\frac{2(n+4s)-2}{(n+4s)+2})$ (such kind of $\nu$ exists since $n\ge 1$ and $s > \frac{1}{2}$), and choose $\frac{\nu+1}{n+4s}<\lambda<\min(\frac{3-\nu}{4},\frac{2-\nu}{2})$. Let $\rho=\varepsilon^{-\lambda}$, we obtain
\begin{equation*}
\varepsilon^{-\nu}\left((\varepsilon^{\nu}|z|+\varepsilon\rho)^2+\rho^{-n-4s}+\varepsilon^{2}\rho^{-n-4s+2}\right)\to 0,
\end{equation*}
and
\begin{eqnarray*}
&&\varepsilon^{-\nu-1}\left((\varepsilon^{\nu}|z|+\varepsilon\rho)^4+\rho^{-n-4s}\right)\to 0.
\end{eqnarray*}
This completes the proof.
\end{proof}

\subsection{The proof of the main result.}

\begin{proof}[Proof of Theorem \ref{t:main}.]
By assumption $0$ is a non-degenerate critical point of $V$, so the image set of $\mathbb{S}: = \partial B_1$ by $v_0$ is diffeomorphic to $\mathbb{S}$. By Lemma \ref{l:uniform}, for sufficiently small $\varepsilon$, $v_{\varepsilon}(\mathbb{S})$ is also diffeomorphic to $\mathbb{S}$. Then there is a point $z_0\in B_1$ such that $v_{\varepsilon}(z_0)=0$. In fact, if not, then for all $z\in B_1$, $v_{\varepsilon}(z)\neq 0$. Let $\tilde v_{\varepsilon}(z)=\frac{v_{\varepsilon}(z)}{|v_{\varepsilon}(z)|}$. Since $v_{\varepsilon}(\mathbb{S})$ is diffeomorphic to $\mathbb{S}$, we have $\tilde v_{\varepsilon}(\mathbb{S})=\mathbb{S}$ and $\tilde v_{\varepsilon}(\overline{B_1})=\mathbb{S}$. By the Brouwer fixed point theorem, it is impossible.

Then, $z_\varepsilon:=\varepsilon^\nu z_0\in B_{\varepsilon^\nu}$ satisfies
$s_{\varepsilon}(z_\varepsilon)=0$.
On other hand, by Lemma \ref{l:contraction2}, we have $S_{\varepsilon}(u_{z,\varepsilon}+\phi_{z,\varepsilon})\in K_{z,\varepsilon}$. So, finally, we obtain $$S_{\varepsilon}(u_{z_\varepsilon,\varepsilon}+\phi_{z_\varepsilon,\varepsilon})=0.$$
Hence (\ref{e:main}) has a solution of form
\begin{equation}
v_\varepsilon (x) = u_{0}\left(\frac{x - z_\varepsilon}{\varepsilon}\right) + \phi_{z_\varepsilon,\varepsilon}
\end{equation}
with $z_\varepsilon \to 0$ and $\|\phi_{z_\varepsilon,\varepsilon}\|_{2s}\to 0$ as $\varepsilon \to 0$.
Recall that $\|\phi_{z_\varepsilon,\varepsilon}\|_{2s}\to 0$ implies that $\phi_{z_\varepsilon,\varepsilon}\to 0$ uniformly.
This completes the proof.
\end{proof}

%\newcommand{\Toappear}{to appear in}

%% \bibliography{mrabbrev,mlabbr2003-0,papers2005,books2005,localbib}
\bibliography{mrabbrev,mlabbr2003-0,localbib}
\bibliographystyle{plain}
\end{document}